\documentclass[12pt]{amsart}
\usepackage{xcolor}
\usepackage{a4,amsmath,amssymb,url}
\usepackage[colorlinks=true,urlcolor=blue,linkcolor=blue,citecolor=blue]{hyperref}

\numberwithin{equation}{section}
\theoremstyle{plain}
\newtheorem{thm}{Theorem}[section]
\newtheorem{lem}[thm]{Lemma}
\newtheorem{cor}[thm]{Corollary}

\theoremstyle{definition}

\newtheorem{exa}[thm]{Example}

\theoremstyle{remark}

\newcommand{\A}{{\mathbf A}}
\newcommand{\B}{{\mathbf B}}

\newcommand{\D}{{\mathbf D}}
\newcommand{\E}{{\mathbf E}}

\newcommand{\G}{{\mathbf G}}

\renewcommand{\P}{\mathbf{P}}

\newcommand{\V}{\mathbf{V}}

\newcommand{\N}{{\mathbb{N}}}
\newcommand{\Z}{{\mathbb{Z}}}

\newcommand{\id}{\mathrm{id}}

\newcommand{\Clo}{\mathrm{Clo}}

\newcommand{\Mlt}{\mathrm{Mlt}}

\newcommand{\st}{\ : \ }%{\,\pmb{|}\,}  
\newlength{\ldprobleft} 
\newlength{\ldprobmid}  
\newlength{\ldprobright}

\newcommand{\pmend}{\color{black}}

\title{Every finite nilpotent loop has a supernilpotent loop as reduct}
\date{\today}
\author{Michael Kompatscher}
\address{Department of Algebra, Charles University, Prague, Czech Republic}
\email{kompatscher@karlin.mff.cuni.cz}
\author{Peter Mayr}
\address{Department of Mathematics,
University of Colorado Boulder, USA}
\email{peter.mayr@colorado.edu}
\thanks{The first author has received funding from the Czech Science Foundation (GAČR grant no. 25-16324S), as well as research grants PRIMUS/24/SCI/008 and UNCE/24/SCI/022 of Charles University. The second author was supported by the National Science Foundation under Grant No. DMS 2452289.}

\begin{document}
\maketitle

\begin{abstract}
 A basic fact taught in undergraduate algebra courses is that every finite nilpotent group is a direct product of
 $p$-groups. Already Bruck~\cite{Br:CTL} observed that this does not generalize to loops.
 In particular, there exist nilpotent loops of size $6$ which are not direct products of loops of size $2$ and $3$.
 Still we show that every finite nilpotent loop $(A,\cdot)$ has a binary term operation $*$
 such that  $(A,*)$ is a direct product of nilpotent loops of prime power order, i.e., $(A,*)$ is supernilpotent. 
 As an application we obtain that every nilpotent loop of order $pq$ for primes $p,q$ has a finite basis for
 its equational theory.
\end{abstract}

\section{Background and results}

 We start by reviewing some concepts and notation of general algebra that we will need. 
 An \emph{algebra} $\A := (A,F)$ is a structure with universe $A$ and a list of basic operations $F$ on $A$.
 For $k\geq 1$, an operation $g\colon A^k\to A$ is a \emph{term operation} of $\A$ if $g$ is a composition of
 basic operations and projections of $\A$.
 More generally, $g\colon A^k\to A$ is a \emph{polynomial operation} of $\A$ if $g$ is a composition of
 basic operations, projections and constants of $\A$.
% For example, $m(x_1,x_2,x_3) := x_1x_2^{-1}x_3$ is a ternary term operation of a group $(A,\cdot,^{-1},1)$.
 An algebra $(A,G)$ is a \emph{reduct} of $\A$ if every operation in $G$ is a term operation of $\A$;
 it is a \emph{polynomial reduct} of $\A$ if every operation in $G$ is a polynomial operation of $\A$. 
% For example, the group $(\{0,1\},+,-,0)$ is a reduct of the Boolean algebra $(\{0,1\},\wedge,\vee,\neg)$.
 Two algebras on the same universe are \emph{term equivalent} if they have the same term functions.

 A \emph{loop} is an algebra $\A = (A,\cdot,\backslash,/,1)$ with binary operations $\cdot,\backslash,/$ and a constant
 $1$ satisfying
\[ x\backslash(xy) = y \quad x(x\backslash y) = y \quad (yx)/x = y \quad  (y/x)x = y \quad x\cdot 1 = 1\cdot x = x. \]
 The left division $\backslash$ and right division $/$ are uniquely determined by the multiplication
 $\cdot$ by these identities. 
 Hence equivalently a loop can be considered as a groupoid $(A,\cdot)$ with identity $1$ such that for all
 $a,b\in A$ the equations $ax=b$ and $ya = b$ have unique solutions $x,y\in A$.
 So for $a\in A$, the left and right translations
\[ L_a\colon A\to A, x \mapsto ax, \quad R_a\colon A\to A, x \mapsto xa, \]
 are bijections. They generate the \emph{multiplication group} $\Mlt(\A)$ of $\A$, that is,
\[  \Mlt(\A) := \langle L_a,R_a \st a\in A\rangle. \]
 In fact finite loops $(A,\cdot,\backslash,/,1)$ are always term equivalent to their reduct $(A,\cdot)$,
 which is why we will not distinguish between the two representations in the following. 
 Groups can be viewed as loops with associative multiplication.
 
 An algebra $\A$ is a \emph{Mal'cev algebra} if it has a ternary \emph{Mal'cev term} operation $m$,
 that is, $m$ satisfies $m(x,y,y) = x = m(y,y,x)$ for all $x,y\in A$.
 In particular, a group $(A,\cdot,^{-1},1)$ has a Mal'cev term $m(x,y,z) := xy^{-1}z$.
 A loop $(A,\cdot,\backslash,/,1)$
 % with multiplication $\cdot$, left division $\backslash$, right division $/$
 has Mal'cev terms $(x/y)z$ and  $x(y\backslash z)$.

 We refer to~\cite{FM:CTC} for the definition and basic properties of the (term condition) commutator $[\alpha,\beta]$
 of congruences $\alpha,\beta$ for an algebra $\A$. This generalizes the classical commutator of normal subgroups in
 groups and allows us, for instance, to talk about abelian, nilpotent and solvable algebras.
 Let $0_A,1_A$ be the trivial congruence (equality) and the total congruence on $\A$, respectively.
 Then $\A$ is $k$-\emph{nilpotent} if there exist a series of congruences
\[ 1_A = \alpha_0 \geq \alpha_1 \geq \dots\geq \alpha_k = 0_A \]
 on $\A$ such that
\[ [\alpha_i,1_A] \leq \alpha_{i+1} \text{ for all } i < k. \]
 As for groups we call such a series of congruences \emph{central}.
 We call $\A$ \emph{nilpotent} if it is $k$-nilpotent for some $k\geq 1$;
 we call $\A$ \emph{abelian} if it is $1$-nilpotent.

 This notion specializes to classical nilpotence for groups (see the discussion in~\cite{FM:CTC})
 and to central nilpotence for loops in the sense of Bruck~\cite{Br:CTL}
 (see Stanovsk\'y and Vojt{\v e}chovsk\'y~\cite{SV:CTL} for a proof).

 Generalizing the binary commutator mentioned above, Bulatov has proposed $k$-ary \emph{higher commutators} in~\cite{Bu:NFM},
 which Aichinger and Mudrinski have used to define another notion of nilpotence in~\cite{AM:SAHC}.
 We refer to their paper and to Moorhead's~\cite{Mo:HCT} for the definition and basic properties of higher commutators.
 An algebra $\A$ is $k$-\emph{supernilpotent} if the $k+1$-ary higher commutator of the total congruence $1_A$ on $\A$
 is the trivial congruence $0_A$, i.e., 
\[ [\underbrace{1_A,\dots,1_A}_{k+1 \text{ times}}] = 0_A. \]
 We call $\A$ \emph{supernilpotent} if it is $k$-supernilpotent for some $k\geq 1$.
 By definition $1$-supernilpotent, $1$-nilpotent and abelian are equivalent for every algebra.

 For groups nilpotence and supernilpotence turn out to be the same but in general they are different.
 A finite Mal'cev algebra $\A$ of finite type is supernilpotent iff it is a direct product of nilpotent algebras
 of prime power cardinality~\cite{AM:SAHC}. In particular a finite loop is supernilpotent iff it is a direct
 product of nilpotent loops of prime power order.
 Furthermore for all Mal'cev algebras $\A$ and all $k\in \N$ the following are equivalent:
\begin{enumerate}   
\item $\A$ is $k$-supernilpotent;
\item  
 $\A$ is $k$-nilpotent and all \emph{commutator terms} on $\A$,
 i.e., term functions satisfying $c(x_1,\dots,x_n,z) = z$ whenever
 $x_i=z$ for some $i\leq n$, have essential arity at most $k+1$.
\end{enumerate}
 Aichinger and Mudrinski prove this equivalence for commutator polynomials
 in~\cite[Corollary 6.15, Lemma 7.5]{AM:SAHC}. But it is not hard to show from known facts that
 on a nilpotent algebra all commutator polynomials have essential arity at most $k+1$ iff all
 commutator terms have essential arity at most $k+1$. The facts needed are:
\begin{itemize}
\item
 Every polynomial $p(x_1,\dots,x_k)$ on $\A$ is of the form $t(x_1,\dots,x_k,a_1,\dots,a_\ell)$ for
 some term function $t$ on $\A$ and constants $a_1,\dots,a_\ell\in A$;
\item
 Every term function $t$ on a nilpotent algebra $\A$ can be represented as ``sum'' of
 commutator terms~\cite[Lemma 2.7]{BM:SPD} (note that although this result states the assumption that $\A$ is finite,
 its proof does not use it).
\end{itemize}  

 Long before the notion of (higher) commutators of congruences existed, Bruck~\cite{Br:CTL} observed that
 there exist (centrally) nilpotent loops of size $6$ which are not direct products of loops of size $2$ and $3$.
 Hence nilpotence does not imply supernilpotence for loops.
 He also showed that for any prime $p$, a loop $\A$ is nilpotent and has $p$-power order iff its multiplication group
 $\Mlt(\A)$ has $p$-power order \cite[Lemma 2.2 of Section VI.2]{Br:SBS}. 
 Wright~\cite{Wr:MGL} generalized this to the result that a ﬁnite loop $\A$
 is supernilpotent iff $\Mlt(\A)$ is nilpotent.
 Recently Seminani{\v s}iniov{\' a} and Stanovsk{\' y} showed that if a loop is $k$-supernilpotent,
 then its multiplication group is $k$-nilpotent~\cite[Theorem 1.1]{SS:TCN}.
 For infinite loops the converse is still open.
 We add that $2$-supernilpotent loops are in fact $2$-nilpotent groups~\cite[Proposition 5.1]{SV:SG3}.

 Unfortunately the relation between nilpotence and supernilpotence for algebras that are not Mal'cev
 is more complicated than their names suggest.
 Kearnes and Szendrei showed that every finite supernilpotent algebra is nilpotent~\cite{KS:ISSN}
 but Moore and Moorhead gave an example of an infinite supernilpotent algebra that is not nilpotent,
 in fact not even solvable~\cite{MM:SNN}.
 
 We can now state our main result in a concise way. 
 
\begin{thm}\label{thm:loop}
 Every finite nilpotent loop has a supernilpotent loop reduct.
\end{thm}

 Theorem~\ref{thm:loop} is proved in Section~\ref{sec:proofs}.
 It gives a positive answer to~\cite[Question 1.5]{Ma:VLN} in the case of loops.
 We discuss some applications to equational theories of nilpotent loops in Section~\ref{sec:applications}.
 In Section~\ref{sec:groups} we show that Theorem~\ref{thm:loop} can neither be extended to arbitrary
 loops nor sharpened to always obtain an abelian reduct.
 
 Every nilpotent loop of order a prime or $4$ is in fact an abelian group
 (however for any odd prime $p$ there exist nilpotent loops of order $p^2$ that are not groups~\cite{DV:ENL}).
 Hence Theorem~\ref{thm:loop} yields in particular that every nilpotent loop
 of order $2^en$ for $e\leq 2$, $n$ odd and squarefree, has an abelian group as reduct.

 Conversely every finite nilpotent loop can be obtained by expanding a supernilpotent loop by some binary
 operation.

 A subloop $N$ of a loop $\A = (A,\cdot)$ is \emph{normal} if for all $a,b\in A$
\[ (ab)N = a(bN) = a(Nb). \]
 If $N$ is normal in $\A$, then the set of left cosets $A/N$ forms a quotient loop $\A/N$ with multiplication 
\[ (aN) \cdot (bN) := (ab)N \]
 for $a,b\in A$.
 
\begin{cor} \label{cor:loop}
 Let $\A = (A,\cdot)$ be a finite nilpotent loop with normal subloop $N$ such that $\A/N$ is supernilpotent. 
 Then $\A$ is term equivalent to an algebra $(A,*,r)$ such that $(A,*)$ is a supernilpotent loop with $(A,*)/N = \A/N$
 and $r$ is a binary operation $r\colon A^2\to N$.
\end{cor}

 Corollary~\ref{cor:loop} is also proved in Section~\ref{sec:proofs}.
 
 Freese and McKenzie observed that every nilpotent Mal'cev algebra $\A$ has polynomial functions
 $\cdot,\backslash,/$
 % (multiplication, left division, right division)
 such that $(A,\cdot,\backslash,/,1)$ is a nilpotent loop (see the discussion after Corollary 7.7 in~\cite{FM:CTC}).
 Hence Theorem~\ref{thm:loop} immediately yields the following.

\begin{cor} \label{cor:Mal'cev}
 Every finite nilpotent Mal'cev algebra has a supernilpotent Mal'cev algebra as polynomial reduct. 
\end{cor}

 The following questions remain open:
\begin{enumerate}
\item
 Is finiteness necessary for Theorem~\ref{thm:loop}, i.e.,
 does every nilpotent loop have a supernilpotent loop reduct?
\item \cite[Question 1.5]{Ma:VLN}
 Can  Corollary~\ref{cor:Mal'cev} be sharpened by replacing ``polynomial reduct'' by ``reduct'', i.e.,
 does every (finite) nilpotent Mal'cev algebra have a supernilpotent Mal'cev reduct?
\end{enumerate}

\section{Proofs of the main results} \label{sec:proofs}

 Let $\A = (A,\cdot,\backslash,/,1)$ be a loop, let $\alpha$ be a congruence of $\A$. 
 Then $\alpha$ is uniquely determined by the normal subloop $N := 1/\alpha$ (the class of $1$) of $\A$,
 more precisely
\[ \alpha = \{ (x,y) \in A^2 \st x/y \in N \} =:\ \equiv_N. \] 
 Conversely, every normal subloop $N$ of $\A$ induces a congruence $\equiv_N$ as above whose classes are just
 the left (or equivalently right) cosets of $N$ in $\A$.

 Because of this correspondence, it is customary in loop theory to denote quotients $\A/\alpha$ by $\A/N$
 for the corresponding normal subloop $N$. It also allows us to translate commutators of congruences $\alpha,\beta$ of
 $\A$ into commutators of normal subloops via
\[ [1/\alpha,1/\beta] := 1/[\alpha,\beta]. \] 
 So $\A$ is nilpotent iff it has a central series of normal subloops
\[ A  = C_0 \geq C_1 \geq \dots \geq C_k = 1 \]
 such that $[C_i,A] \leq C_{i+1}$ for $i<k$.
 It is not hard to see that whenever such a central series exists for some finite $\A$,
 it can be refined to one where all consecutive quotients $C_i/C_{i+1}$ have prime power order.
 
 Our main result will follow readily from the next lemma.

\begin{lem} \label{lem:loop}
 Let $\A$ be a finite nilpotent loop with a central series of normal subloops
\[ A  = C_0 > C_1 > \dots > C_k = 1 \]
 and factors $C_i/C_{i+1}$ of prime power order for $i<k$.

 Then for every $i\leq k$ there exists a binary term operation $*_i$ on $\A$ such that $\A_i := (A,*_i)$ is a loop and
 $\A_i/C_i$ is supernilpotent.
\end{lem}
 
\begin{proof}
 Let us start with a comment on notation.
 Recall that for any $j\leq k$, the normal subloop $C_j$ induces a congruence $\alpha_j$ on the loop $\A$.
 Clearly $\alpha_j$ is still a congruence for any reduct $\A_i$ of $\A$ that we will construct in this proof.
 Further $\A_i/\alpha_j$ is a reduct of $\A/\alpha_j = \A/C_j$ and has the same universe $A/\alpha_j = A/C_j$.
 Hence we will denote $\A_i/\alpha_j$ by $\A_i/C_j$ for simplicity throughout this proof.
 Since all our reducts $\A_i$ will turn out to be loops with normal subloops $C_j$,
 this is actually consistent with the usual notation in loops.

 We will prove the lemma by induction on $i$. For $i=1$, let $*_1$ be the multiplication of $\A = (A,\cdot)$.
 Then $\A_1 = \A$ and $\A/C_1$ is abelian of prime power order by assumption, in particular, $1$-supernilpotent.

 Next let $i\geq 1$ and assume $\A_i = (A,*_i)$ is a loop reduct of $\A$ such that $\A_i/C_i$ is supernilpotent for
 $i\geq 1$.
 If not denoted otherwise, all loop operations are those from $\A_i$ in the following. 
 
 Let $|C_i/C_{i+1}|$ be a power of a prime $p$. Since $\A_i/C_i$ is supernilpotent by induction assumption,
 it is a direct product of loops of prime power order. In particular $\A_i/C_i \cong \P \times \V$ where $|P|$
 is a power of $p$ and $|V|$ is coprime to $p$.
 By the Homomorphism Theorem $\A_i/C_{i+1}$ has a unique maximal $p$-power order subloop $\E$,
 which is the extension of $C_i/C_{i+1}$ by $\P$.
 Our goal is to define a new loop reduct $\A_{i+1}$ such that
\begin{equation} \label{eq:Ai+1}
 \A_{i+1}/C_{i+1} \cong \E\times\V.
\end{equation}
 Since the latter is supernilpotent as a direct product of nilpotent loops of prime power order,
 this will establish the induction step and hence the lemma.

 First note that there exists $n\geq 1$ such that
\[ x^n = x \text{ for all } x\in E \text{ and } x^n = 1 \text{ for all } x\in V. \] 
 Here $x^n := (..((\underbrace{xx)x) \dots)x}_{n}$ denotes the left associated $n$-th power of $x$ with respect to $*_i$.
 The natural number $n$ as above exists since for any nilpotent loop of order a power of a prime $q$,
 the group of right translations is a $q$-group \cite[Lemma 2.2 of Section VI.2]{Br:SBS}. 
 So we can choose $n$ to be congruent to $1$ modulo the order of the multiplication group of $\E$ (which is a $p$-group)
 and $n\equiv 0$ modulo the order of the multiplication group of $\V$ (which is a direct product of $q$-groups, for primes $q \neq p$).

 We write $\bar{x}:= xC_{i+1}$ for $x\in A$, $\bar{A} := A/C_{i+1}$ and $C := C_i/C_{i+1}$. Let
\[ h\colon \bar{A}\to\bar{A},\ \bar{x}\mapsto\bar{x}^n. \]
 Our choice of $n$ yields that $h$ modulo $C$ just induces the projection homomorphism of $\bar{\A}/C \cong \P \times \V$ onto $\P$.
 In particular $h(\bar{A}) \subseteq E$ and also $h|_E = \id_E$. Hence $h$ behaves like a projection of $\bar{A}$ onto $E$
 but it need not be a homomorphism of $\bar{\A}_i = (A/C_{i+1},*_i)$.
 For $x,y\in A$ define
\[ r(x,y) := (xy)^{n} / (x^{n}y^{n}). \]
Modulo $C_{i+1}$ this term is equal to the quotient of $h(\overline{xy})$ and $h(\bar{x})h(\bar{y})$,   i.e.
\[ h(\overline{xy}) = \overline{r(x,y)}\, \bigl(h(\bar{x})h(\bar{y})\bigr) \text{ for all } x,y\in A. \]
 Since $h$ modulo $C$ is a homomorphism,
\[ \overline{r(x,y)} \in C \text{ for all } x,y\in A. \]
 Now define $\A_{i+1} := (A,*_{i+1})$ with
 \[ x*_{i+1} y := (xy)/r(x,y). \]
% Note that if $\bar{\A}_i \cong \E\times\V$ is already supernilpotent,
% then $h\colon \bar{\A}_i\to \E$ is the projection homomorphism and $\overline{r(x,y)} = \bar{1}$ for all $x,y\in A$.
% Consequently $*_i$ and $*_{i+1}$ induce the same multiplication modulo $C_{i+1}$ but may still differ on $A$.
%
 We write $\bar{\A}_{i+1}$ for the reduct of $\bar{\A}$ with universe $A/C_i$ and operation induced by $*_{i+1}$.
 Since $r$ clearly vanishes on $E$ and on $V$ as well, the new $*_{i+1}$ and the old $*_i$ induce the same operations
 on $E$ and on $V$.
 In particular $\bar{\A}_i/E \cong \V \cong \bar{\A}_{i+1}/E$.
 To see that $\bar{\A}_{i+1}$ also has a quotient isomorphic to $\E$,
 we check that $h\colon \bar{\A}_{i+1} \to \E$ is a homomorphism. 
 For that let $x,y\in A$ and consider
\begin{align*}
  h(\bar{x}*_{i+1} \bar{y}) & = \left(\overline{xy}/\overline{r(x,y)}\right)^n \\
                            &= \overline{xy}^n/\overline{r(x,y)}^n  & \text{by the centrality of }\overline{r(x,y)} \text{ in } \bar{\A}_i \\
                            &= \overline{xy}^n/\overline{r(x,y)}  & \text{by the definition of } n \text{ and } \overline{r(x,y)}\in E \\
                            &= h(\bar{x}) h(\bar{y}) & \text{by definition of } r(x,y) \\
 &= h(\bar{x}) *_{i+1} h(\bar{y}) & \text{since } *_i = *_{i+1} \text{ on } h(\bar{A}) = E.
\end{align*}  
 Since $\bar{\A}_{i+1}/\ker h \cong \E$, we see that $\bar{\A}_{i+1}$ has coprime quotients $\E$ and $\V$. 
 Hence~\eqref{eq:Ai+1} follows once we show that $\A_{i+1}$ is a loop. 

 For this we use a second induction on $j = i, \ldots, k$ to prove that
\begin{equation} \label{eq:AiCj}
  \A_{i+1}/C_j := (A/C_j,*_{i+1}) \text{ is a loop}.
\end{equation}
 The base case $j=i$ follows from our previous observation that $*_{i+1}$ and $*_i$ coincide on $E$ and on $V$.
 In particular $\A_{i+1}/C_i = \A_{i}/C_i$ is a loop.

 Next let $j\geq i$ and assume $\A_{i+1}/C_j$ is a loop. We claim that the right multiplication by any fixed $b\in A$,  
\begin{equation} \label{eq:bij}
 A/C_{j+1} \to A/C_{j+1},\ xC_{j+1}\mapsto (x*_{i+1} b) C_{j+1}, \text{ is a bijection}.
\end{equation}  
 For the proof consider $a,a'\in A$ with $a*_{i+1} b \equiv_{C_{j+1}} a'*_{i+1} b$.
 Since $\A_{i+1}/C_j$ is a loop by induction assumption, this yields $a \equiv_{C_j} a'$.
Hence we have $c\in C_j$ such that $a'=a\cdot c$ (using the original multiplication $\cdot$ of $\A$).

Define the term $p(x,y,z) := (x\cdot y) *_{i+1} z$ on $\A$.  By the assumption above we have
$$p(a,1,b) \equiv_{C_{j+1}} p(a,c,b).$$
Since $C_j/C_{j+1}$ is central in $\A/C_{j+1}$, the term condition characterization
of centrality \cite[Definition 3.2.]{FM:CTC} yields that
\[ 1 = p(1,1,1) \equiv_{C_{j+1}} p(1,c,1) = c. \]
In other words $c \in C_{j+1}$ and hence $a \equiv_{C_{j+1}} a'$.
% Let $m(x,y,z) := (x/y)z$ be a Mal'cev term for $\A$. Since $C_j/C_{j+1}$ is central in $\A/C_{j+1}$ and
% $a \equiv_A 1 \equiv_{C_j} c$, $b \equiv_A 1 \equiv_{C_j} 1$,
% \cite[Proposition 5.7]{FM:CTC} yields that
%\[ a'*_{i+1}b = m(a,1,c) *_{i+1}m(b,1,1) \equiv_{C_{j+1}} m(a*_{i+1}b,1*_{i+1} 1, c*_{i+1}1) = (a*_{i+1}b)c. \]
%Together with the assumption $a*_{i+1} b \equiv_{C_{j+1}} a'*_{i+1} b$ this yields $c\in C_{j+1}$,
%hence $a \equiv_{C_{j+1}} a'$.
Thus the map in~\eqref{eq:bij} is injective. Bijectivity follows from the finiteness of $A$.

 Similarly the left multiplication by any fixed element modulo $C_{j+1}$ is bijective.
 Thus~\eqref{eq:AiCj} is proved.
 For $j=k$ we obtain that $\A_{i+1}$ is a loop.

 Summing up, $\A_{i+1}/C_{i+1}$ is a loop of order $|E|\cdot |V|$ with coprime quotients $\E$ and $\V$. 
 Thus $\A_{i+1}/C_{i+1} \cong \E\times\V$ and~\eqref{eq:Ai+1} is proved.
 In particular $\A_{i+1}/C_{i+1}$ is supernilpotent. 
\end{proof}

\begin{proof}[Proof of Theorem~\ref{thm:loop}]
 For $i=k$ in Lemma~\ref{lem:loop} we see that $\A = (A,\cdot)$ has a supernilpotent loop reduct
 $\A_k = (A,*_k) \cong \A_k/1$.
\end{proof}

\begin{proof}[Proof of Corollary~\ref{cor:loop}]
 We may choose a central series of normal subloops of $\A = (A,\cdot)$,
\[ A  = C_0 > C_1 > \dots > C_k = 1, \]
 such that factors $C_i/C_{i+1}$ have prime power order for all $i<k$ and $C_j = N$ for some $j \leq k$.
 Since $\A/N$ is supernilpotent, we may start the inductive construction in the proof of Lemma~\ref{lem:loop}
 with $\A_j = \A$ instead of $\A_1 = \A$.
 This then yields $\A_i/N = \A/N$ for all $i\leq k$. Let $/_k$ denote the right division in $\A_k$, and let
\[ r\colon A^2\to A,\ (x,y)\mapsto (xy)/_k (x*_ky). \]
 Then the range of $r$ is contained in $N$. The loop $\A_k := (A,*_k)$ is supernilpotent
 and its expansion $(A,*_k,r)$ is term equivalent to $\A$.
\end{proof}

\section{An application to term functions and equational theories} \label{sec:applications}

 Vaughan-Lee~\cite{Va:NPV} showed that every finite supernilpotent loop has a finite basis of equations for its
 equational theory (see~\cite[Chapter 14]{FM:CTC} for a generalization to Mal'cev algebras).
 It remains open whether any finite nilpotent loop is finitely based.
 We hope that Theorem~\ref{thm:loop} can serve as a first step towards a possible positive answer.

 To illustrate this, we give an explicit description of term functions on finite central-by-supernilpotent
 (in particular $2$-nilpotent) loops.
 In~\cite{Ma:VLN} this approach was already used to show that a particular nilpotent (but not supernilpotent)
 loop of size 12 that was proposed by Vaughan-Lee in~\cite{Va:NPV} as potentially non-finitely based
 is in fact finitely based.
 In Theorem~\ref{thm:eqbasispq} we show that all nilpotent loops of size $pq$ for primes $p, q$ are finitely based.

 For an algebra $\A$, let $\Clo(\A)$ denote the set of term functions of $\A$.
 For algebras $\A,\B$ and a function $r\colon A^k\to B$, let $\langle r\rangle_{\A,\B}$ be the \emph{clonoid}
 from $\A$ to $\B$ generated by $r$, i.e., the smallest set of finitary functions from $A$ to $B$ that contains $r$
 and is closed under compositions with term functions of $\A$ on the domain and with term functions of $\B$
 on the codomain side (see~\cite{MW:CBM} for the definition and background on clonoids). 
 
\begin{lem} \label{lem:centralbysn}
 Let $\A = (A,\cdot)$ be a finite loop with identity $1$ and central normal subloop $N$ such that
 $\A/N$ is supernilpotent.
\begin{enumerate}
\item  \label{it:csn1} 
 Then $\A$ is term-equivalent to $(A,*,r)$ for some supernilpotent loop $(A,*)$
 with $\A/N = (A,*)/N$, $(N,\cdot) = (N,*)$ 
 and some $r\colon A^2\to N$ that is constant on cosets modulo $N$ and $r(1,1)=1$.
\item \label{it:csn2}
 $\Clo(\A) = \{ f*w \st f\in\Clo(A,*), w\in\langle r\rangle_{\A/N,(N,\cdot)} \}. $

 Here we abuse notation by identifying functions $r,w$ from $A$ to $N$ that are constant on cosets of
 $N$ with their induced functions from $A/N$ to $N$.
\end{enumerate}  
\end{lem}

\begin{proof}
 \eqref{it:csn1}
 As in the proof of  Corollary~\ref{cor:loop} we obtain a supernilpotent loop reduct $(A,*)$
 with right division $/$ of $\A = (A,\cdot)$ such that $\A/N = (A,*)/N$ and
\[ r\colon A^2\to N,\ (x,y)\mapsto (xy)/(x*y). \]
 Clearly $(A,*,r)$ and $\A$ are term equivalent and $r(1,1)=1$.
 It remains to show that $\cdot$ and $*$ are equal on $N$ and that $r$ is constant on the cosets modulo $N$.

 Note that $N$ is central in $(A,\cdot)$ and in $(A,*)$. Further it induces the same congruence $\alpha$ on both loops
 with classes of the form $xN = x*N$ for $x\in A$. 

 Let $m(x,y,z) := (x/y)*z$ be a Mal'cev term for $(A,*)$ (and hence for $(A,\cdot)$). Let $c,d\in N$. Then
 \cite[Proposition 5.7]{FM:CTC} yields that
\[   c*d = m(c\cdot 1, 1\cdot 1, 1 \cdot d) = m(c,1,1) \cdot m (1,1,d) = c\cdot d. \]
 Hence $(N,\cdot) = (N,*)$. Further for $a,b\in A$, we get
\[ r(a*c,b*d) = r\bigl(m(a,1,c),m(b,1,d)\bigr) = m\bigl( r(a,b), \underbrace{r(1,1)}_{=1}, \underbrace{r(c,d)}_{=1} \bigr) = r(a,b). \]
% \[ (a*c)\cdot(b*d) =  m(a,1,c) \cdot m(b,1,d) = m(ab,1,cd) = (ab)*(cd). \]
%  Since $N$ is central in $(A,*)$, clearly
% \[ (a*c)*(b*d) = (a*b)*(c*d). \] 
%  Together the previous three equations yield
%  \begin{align*}
%    r(a*c,b*d) & = [(a*c)\cdot(b*d)]/[(a*c)*(b*d)] \\
%               & = [(ab)*(cd)]/[(a*b)*(c*d)] \\
%               & = (ab)/(a*b) \\
%               & = r(a,b).
%  \end{align*}
 Hence $r$ is constant on the cosets modulo $N$ and~\eqref{it:csn1} is proved.

 \eqref{it:csn2}
 Let $D := \{ f*w \st f\in\Clo(A,*), w\in\langle r\rangle_{\A/N,(N,\cdot)} \}$.
 
 The inclusion $\Clo(\A)\supseteq D$ follows from $*,r\in\Clo(\A)$.
 For the converse note that $D$ contains all projections on $A$. It remains to show that $D$
 is closed under $*$ and $r$. For this let $f_1,f_2\in\Clo(A,*)$ and $w_1,w_2\in\langle r\rangle_{\A/N,(N,\cdot)}$.
 Then
\[ (f_1*w_1)*(f_2*w_2) = (f_1*f_2) *(w_1*w_2) \in D_k \]
 since the image of $w_1,w_2$ is contained in $N$, which is central in $(A,*)$.
%Since $(N,*) = (N,\cdot)$
% by~\eqref{eq:N}, we see that $(f_1*f_2) *(w_1*w_2)\in D_k$.
 Next
\[ r(f_1*w_1,f_2*w_2) = r(f_1,f_2) \in \langle r\rangle_{\A/N,(N,\cdot)} \] 
 since $r$ is constant on cosets modulo $N$. Thus $\Clo(\A)\subseteq D$.
\end{proof}

 Let $\A$ be a loop as in Lemma~\ref{lem:centralbysn}. Since $\A$ and $(A,*,r)$ are term equivalent, $\A$ has a finite equational basis if and only if $(A,*,r)$ has one. We claim now that this is already the case if the clonoid $\langle r\rangle_{\A/N,(N,\cdot)}$ is finitely based (or, more precisely, the multi-sorted algebra with sorts $A/N$ and $N$ and operations $\cdot$ on $A/N$ and on $N$ as well as $r\colon (A/N)^2 \to N$).
 
Here, an equational basis of $(A,*,r)$ can be obtained as the union of a finite equational basis of the supernilpotent loop $(A,*)$ (which exists by \cite[Theorem 14.16.]{FM:CTC}) and a list of identities corresponding to the equational basis of $\langle r\rangle_{\A/N,(N,\cdot)}$. We do not include a formal proof of this statement, to avoid the technicalities that arise from discussing equational theories of multi-sorted algebras (see e.g. \cite{Ta:MSA}). But, to illustrate the power of this approach, we are going to show that all nilpotent loops of size $pq$ for primes $p,q$ are finitely based. Besides Lemma \ref{lem:centralbysn}, our proof relies on the analysis of clonoids in \cite{Fi:CFF}.

\begin{thm} \label{thm:eqbasispq}
Every nilpotent loop of order $pq$ for primes $p,q$ has a finite equational basis.
\end{thm}

\begin{proof}  
Let $\A$ be a nilpotent loop of order $pq$ for primes $p,q$. If $p = q$ or $\A$ is abelian, then $\A$ has a finite
equational basis by~\cite[Theorem 14.16.]{FM:CTC}). So assume that $p \neq q$ and that the center $N$ of $\A$ has
size $p$.
%Without loss of generality, we may assume that $\A$ has a central normal subloop $N$ with $N = \Z_p$ and
%$\A/N = \Z_q$.

By Lemma~\ref{lem:centralbysn} we may assume that
$\A$ is term equivalent to $(\Z_q\times\Z_p,+,r)$ with $r\colon A^2 \to N$ constant on cosets of $N$ and $r(0,0)=0$.
In the following, we will abuse notation and use the same symbol for
$r\colon A^2 \to N$ and the function $r\colon \Z_q^2 \to \Z_p$ induced by it.
Then $C := \langle r\rangle_{\A/N,(N,\cdot)}$ is a clonoid of $0$-preserving functions from $(\Z_q,+)$ to $(\Z_p,+)$.
By~\cite[Theorem 1.2]{Fi:CFF} every such clonoid is already generated by a single unary function.
Hence $\A$ is term equivalent to $(\Z_q\times\Z_p,+,f)$ for some unary $f\colon \Z_q \to \Z_p$ satisfying $f(0)=0$.
Our first goal is to show that $C$ (or, more precisely, the multi-sorted algebra with domains $\Z_q$, $\Z_p$ and
operations $+^{\Z_q}$, $+^{\Z_p}$ and $f\colon \Z_q \to \Z_p$) is finitely based.

We denote the $n$-ary functions of $C$ by $C^{[n]}$. Then $C^{[n]}$ is a subspace of $\Z_p^{\Z_q^n}$ for every $n\geq 1$. We claim that
\begin{enumerate}
\item
 there exists $a \in \Z_q \setminus \{0\}$ of multiplicative order $k|q-1$ such that  
 $B := \{ f(a^ix) \st 0\leq i\leq k-1 \}$ is a basis of $C^{[1]}$ and
% for some $a \in \Z_q \setminus \{0\}$ of multiplicative order $k$ and %$k \mid q-1$,
\item
 $\dim(C^{[n]}) = k \cdot \frac{q^{n}-1}{q-1}$.
\end{enumerate}
(1) was essentially observed in \cite[Lemma 5.3]{Fi:CFF}. To verify (2), let $\mathcal L$ be the set of all
$1$-dimensional subspaces of $\Z_q^n$. It follows from \cite[Lemma 4.1]{Fi:CFF} that for every $h \in C^{[1]}$
and every $L \in \mathcal{L}$ generated by $\mathbf{y} \in\Z_q^n$, there exists a function $h_L \in C^{[n]}$
such that $h_L(\lambda \mathbf y) = h(\lambda)$ for all $\lambda\in\Z_q$ and
$h_L(\mathbf x) = 0$ for all $\mathbf x\in\Z_q^n \setminus L$.
Moreover, as in \cite[Proof of Proposition 4.3.]{Fi:CFF}, one can see that $g \in C^{[n]}$ iff $g$ is the sum of
functions of the form $h_L$ for $h \in C^{[1]}$, $L \in \mathcal L$. Combining these two observations,
it follows that $\{h_L \st h \in B, L \in \mathcal L\}$ forms a basis of $C^{[n]}$.
Hence $\dim(C^{[n]}) = |B| \cdot |\mathcal L| = k \cdot \frac{q^{n}-1}{q-1}$ and (2) is proved.

By (1), for every $c\in \Z_q$, there are coefficients $d_0,\ldots,d_{k-1} \in \Z_p$ such that 
\begin{align}
f(cx) = \sum_{i=0}^{k-1} d_i f(a^ix). \label{eq:basis}
\end{align}
We claim that all identities of the form \eqref{eq:basis} together with the identities holding in
$\Z_q$ and $\Z_p$ form an equational basis of the clonoid $C$.

To see this, recall that $C^{[n]}$ is the linear closure of all functions $f(\sum_{i=1}^n \alpha_i x_i)$ for
$\alpha = (\alpha_1,\alpha_2,\ldots,\alpha_n) \in \Z_q^k \setminus \{0\}$. In fact, we only need
coefficients from the set $A_n := \{(\alpha_1,\ldots,\alpha_n) \in\Z_q^k\st \alpha_1= \ldots, \alpha_{i-1} = 0, \alpha_i = a^j $ for some $i\leq n, 0 \leq j < k\}$ by the identities \eqref{eq:basis}.
Since $|A_n| = k \cdot \frac{q^{n}-1}{q-1}$, we get from (2) that the set
$B_n := \{f(\sum_{i=1}^n \alpha_i x_i) \st \alpha \in A_n\}$ is a basis of $C^{[n]}$.

Thus any $n$-ary operation $t \in C$ is equal to a unique linear combination of elements from $B_n$;
on a syntactic level, we can think about this as a normal form of $t$ as a term operation of the multi-sorted algebra given by $f$, and the addition on $\Z_p$ and $\Z_q$. Moreover, any term can be rewritten into such a normal form
by the identities \eqref{eq:basis} and finitely many identities axiomatizing $\Z_p$ and $\Z_q$. So, in this sense, $C$ has a finite equational basis.

In order to obtain an equational basis for $(\Z_q\times\Z_p,+,f)$ note that every term operation $t(x_1,\ldots,x_n)$ of $(\Z_q\times\Z_p,+,f)$ can be uniquely written as $\sum_{i = 1}^n d_i x_i + t'(x_1,\ldots,x_n)$,
where $0\leq d_i <pq$ for all $i\leq n$ and $t'$ is the normal form of an element of $C$.
Thus, any set of identities that allows us to rewrite terms over $(\Z_q\times\Z_p,+,f)$ into such normal forms is an equational basis. It is easy to see that the finite set $\Sigma$ given by
\begin{enumerate}
\item a finite equational basis of $(\Z_q\times\Z_p,+)$,
\item $f(x + f(y)) = f(x)$,
\item $f(x+qy) = f(x)$,
\item $p f(x) = 0$,
\item all identities of the form \eqref{eq:basis}
\end{enumerate}
forms such an equational basis. As $\A$ and  $(\Z_q\times\Z_p,+,f)$ are term equivalent,
this finishes the proof.
\end{proof}

Note that, in the construction of the equational basis $\Sigma$ in the proof of Theorem \ref{thm:eqbasispq},
 the identities in (1) axiomatize $(A,*)$, while the identities (2),(3) and (4) correspond to the identities satisfied by the domain $\A/N$ and codomain $(N,\cdot)$ of the clonoid $C = \langle r \rangle_{\A/N,(N,\cdot)}$ respectively, and (5) concerns all actual non-trivial multi-sorted identities from the equational basis of the clonoid $C$.
\pmend

In a very similar way, we can obtain an equational basis of $(A,*,r)$ in the general situation of Lemma \ref{lem:centralbysn} from an equational basis of $\langle r \rangle_{\A/N,(N,\cdot)}$, together with finite equational bases of $(A,*)$, $\A/N$ and $N$.

\section{Reducts of groups}  \label{sec:groups}

In view of Theorem~\ref{thm:loop} it is natural to ask
\begin{enumerate}
\item whether also non-nilpotent loops may have supernilpotent reducts and
\item when nilpotent loops have abelian reducts? 
\end{enumerate}
 We give answers in the setting of groups starting with the first question.

\begin{lem}
 No finite non-nilpotent group has a supernilpotent loop (or nilpotent group) as reduct.
\end{lem}

\begin{proof}
 Let $\G$ be a finite group with supernilpotent loop reduct $(G,*)$.  
% Since $*$ is a term function of $\A$, it preserves every subgroup of $\A$.
% In particular,
 Then every subgroup of $\G$ is preserved by $*$ and hence also a subloop of $(G,*)$.
 Since $(G,*)$ is supernilpotent, it has a unique subloop $P$ of $p$-power order with $|G|/|P|$ coprime to $p$
 for any prime $p$.
% Sylow $p$-subgroup for any prime $p$.
 Hence $P$ is the unique Sylow $p$-subgroup of $\A$ for any prime $p$. Thus $\A$ itself is nilpotent. 
\end{proof}

Any $2$-nilpotent group $(G,\cdot)$ of odd order has an abelian reduct by the following construction
due to Baer~\cite[Theorem B.1]{Ba:GAC} and called the ``Baer trick'' in~\cite[Lemma 4.37]{Is:FGT}.
 For such a group, the squaring map $s\colon G\to G,\ x\mapsto x^2$, is a bijection with inverse $x\mapsto x^k$
 for some $k\in \N$.
% For $x\in G$, we let $x^{1/2}$ denote the unique element $y\in G$ such that $y^2 = x$.
% For $\G$ a $2$-nilpotent group of odd order,
 It is then not hard to see that $(G,*)$ with
\[ x*y := xy\, [y,x]^{k} \] 
is an abelian group reduct of $\G$.
%This is called the ``Baer trick'' by Isaacs~\cite[Lemma 4.37]{Is:FGT}.
 The construction clearly fails if $[y,x]$ has no pre-image under $s$, in particular,
 if the derived subgroup of $\G$ has even exponent.
% is a non-abelian $2$-nilpotent $2$-group.
 In that case we show that $\G$ does not have any (polynomial) abelian group reduct at all.
 This implies that even in the group case we cannot strengthen Theorem~\ref{thm:loop} to always
 obtain abelian reducts.

 We start with a general characterization of Mal'cev reducts of $2$-nilpotent groups.

\begin{lem} \label{lem:2nil}
 Let $\G = (G,\cdot)$ be a $2$-nilpotent group.
\begin{enumerate}
\item \label{it:2nil1}
 Then $m$ is a Mal'cev polynomial of $\G$ iff there exists some $c\in\Z$ such that
 \[ m(x,y,z) = xy^{-1}z\, \left([x,y]\, [x,z]^{-1}\, [y,z]\right)^c  \text{ for all } x,y,z\in G. \]
 In particular, every Mal'cev polynomial is a Mal'cev term.
\item \label{it:2nil2}
 Every polynomial loop reduct of $\G$ is isomorphic to some group reduct of $\G$. 
 Further $(G,*)$ is a group reduct of $\G$ iff there exists some $c\in\Z$ such that
 \[ x*y = xy [x,y]^c  \text{ for all } x,y\in G. \]
\end{enumerate} 
\end{lem}

\begin{proof}
 \eqref{it:2nil1} Assume $m(x_1,x_2,x_3)$ is a Mal'cev polynomial of $\G$ using constants $g_1,\dots,g_n\in G$.
 Since $\G$ is $2$-nilpotent, commutators are central and we can write $m$ in the form
\begin{equation} \label{eq:mpol1}
 m(x_1,x_2,x_3) = dx_1^{a_1} x_2^{a_2} x_3^{a_3}\, \prod_{i=1}^3\prod_{j=1}^n [x_i,g_j]^{b_{ij}}\, \prod_{1\leq i < j \leq 3} [x_i,x_j]^{c_{ij}}
\end{equation}
 for some integers $a_i, b_{ij}, c_{ij}$ and a constant $d\in G$.

 From $m(1,1,1)=1$ we obtain $d=1$.
 Using the Mal'cev identities we see
\begin{align*}
 & x_1 = m(x_1,1,1) = x_1^{a_1} \prod_{j=1}^n [x_1,g_j]^{b_{1j}}, \\
  & x_3 = m(1,1,x_3) = x_3^{a_3} \prod_{j=1}^n [x_3,g_j]^{b_{3j}}.
\end{align*}
 Replacing the corresponding factors in~\eqref{eq:mpol1}, we obtain the simplification
\begin{equation} \label{eq:mpol2}
  m(x_1,x_2,x_3) = x_1 x_2^{a_2} x_3\, \prod_{j=1}^n [x_2,g_j]^{b_{2j}}\, \prod_{1\leq i < j \leq 3} [x_i,x_j]^{c_{ij}}.
\end{equation}
 By the Mal'cev identity again we have
\[ 1 = m(x_2,x_2,1) = x_2 \underbrace{x_2^{a_2}\, \prod_{j=1}^n [x_2,g_j]^{b_{2j}}}_{=x_2^{-1}}. \]
 Hence we can reduce~\eqref{eq:mpol1} further to 
\begin{equation} \label{eq:mpol3}
  m(x_1,x_2,x_3) = x_1 x_2^{-1} x_3\, \prod_{1\leq i < j \leq 3} [x_i,x_j]^{c_{ij}}.
\end{equation}
 Finally we see from
\[ x = m(x,y,y) = x [x,y]^{c_{12}+c_{13}}\quad \text{ and } \quad x = m(y,y,x) = x [y,x]^{c_{13}+c_{23}} \]
 that we may choose $c_{12} = -c_{13} = c_{23}$.
 Writing $c$ for $c_{13}$ yields the form of $m$ claimed in~\eqref{it:2nil1}.

 \eqref{it:2nil2} Let  $*_e,/,\backslash,e$ be polynomial functions on $\G$ such that $\A := (G,*_e,/,\backslash)$
 is a loop with identity $e$. Then $(x/y)*_ez$ is a Mal'cev polynomial of $\A$ and hence of $\G$. So by~\eqref{it:2nil1}
 we have $(x/y)*_ez = xy^{-1}z\, \left([x,y]\, [x,z]^{-1}\, [y,z]\right)^c$
 for some $c\in\Z$. Setting $y=e$ we see that
\[ x*_ez = (x/e)*_ez = xe^{-1}z\, \left([x,e][x,z]^{-1}[e,z]\right)^c \text{ for all } x,z\in G. \]
 For $e=1$, the identity of $\G$, this simplifies further to
\[ x*_1 z = xz\, [x,z]^{-c}. \]
 Using that $\G$ is $2$-nilpotent and commutators are bilinear,
 it is now straightforward to check that $(G,*_1)$ is associative and hence a group.
 Further $\varphi\colon (G,*_1) \to (G,*_e),\ x\mapsto xe$, is easily seen to be an isomorphism.
 Since $(G,*_1)$ is a group, so is $(G,*_e)$. 
 Thus every polynomial loop reduct of $\G$ is a group.
 Since $1$ is the only constant term function of $\G$, every (term) loop reduct of $\G$ is a group with
 multiplication $*_1$.
\end{proof}

\begin{cor} \label{cor:2nil}
 A $2$-nilpotent group $\G$ has an abelian group as (polynomial) reduct iff the derived subgroup of $\G$
 has odd exponent.
\end{cor}

\begin{proof}
 By Lemma~\ref{lem:2nil} every polynomial group reduct of $\G = (G,\cdot)$ is isomorphic to $(G,*)$ with
 \[ x*y = xy\, [x,y]^c \]
 for some $c\in\Z$. Now $(G,*)$ is abelian iff
\[ xy\, [x,y]^c = yx\, [y,x]^c. \]
 The latter is equivalent to
\[  [x,y]^{2c+1} = 1 \text{ for all } x,y\in G. \]
 Clearly such an integer $c$ exists iff the derived subgroup of $\G$ has odd exponent (possibly $1$).
\end{proof}

\begin{exa}
% The symmetric group on $3$ letter is the smallest group without nilpotent group reduct.
 Let $\G = (G,\cdot)$ be a group with derived subgroup of exponent $2$ (e.g., any non-abelian group of order $8$). 
% The non-abelian groups $\G = (G,\cdot)$ of order $8$ are the minimal nilpotent groups without abelian reduct.
% Recall that the derived subgroup of such a $\G$ has exponent $2$. Hence
 By Lemma~\ref{lem:2nil}
\begin{enumerate}
\item every loop reduct of $\G$ is equal to $\G$ or $\G^{op} = (G,*)$ with $x*y = yx$,
\item every polynomial loop reduct of $\G$ is isomorphic to $\G$,
\item every Mal'cev polynomial of $\G$ is equal to $xy^{-1}z$ or $zy^{-1}x$.
\end{enumerate}
 We further remark that $(G,xy^{-1}z)$ is a \emph{minimal} Mal'cev algebra,  meaning that it has a Mal'cev term,
 but does not have any proper reducts with a Mal'cev term.  This directly implies that it is also a minimal Taylor
 algebra in the sense of \cite{BBBKZ:MT} since for finite solvable algebras having a Taylor term is equivalent to
 having a Mal'cev term by results from tame congruence theory \cite{HM:SFA}.
\end{exa}

 We give two more examples showing that in general it is not possible to decrease the nilpotence class
 of a $2$-group or of a group of class greater than $2$ by taking a group reduct:
 
 For $n\geq 2$, let $\D_{2^{n+1}}$ be the dihedral group of order $2^{n+1}$. 
 Then $\D_{2^{n+1}}$ has a $2$-element center and is $n$-nilpotent but not $n-1$-nilpotent.
 By Lemma~\ref{lem:2nil} and induction on $n$ it follows that every group reduct of $\D_{2^{n+1}}$ has the same properties,
 in particular, is not $k$-nilpotent for $k<n$.

The wreath product $\Z_3\wr\Z_3$ is nilpotent of class $3$.
%A straightforward analysis of its term functions shows that $\Z_3\wr\Z_3$ does not have an abelian group reduct.
We claim that $\G$ has no abelian group reduct $(G,*)$. Suppose otherwise. Then every $g\in G$ generates the
same subgroup and hence has the same order in $\G$ and in $(G,*)$. Since $\G$ has exponent $9$ and a normal
subgroup $N$ isomorphic to $\Z_3^3$, this implies that $(G,*)$ is isomorphic to $\Z_9\times\Z_3^2$.
Hence every $g\in G\setminus N$ has order $9$, which is clearly not true in $\G$. So $\G$ has no abelian reduct.
 By Corollary~\ref{cor:2nil} it does not have a $2$-nilpotent group reduct either.

\bibliographystyle{plain}
%\bibliography{../../biblio}

\begin{thebibliography}{1}
\bibitem{AM:SAHC}
E.~Aichinger and N.~Mudrinski.
\newblock Some applications of higher commutators in {M}al'cev algebras.
\newblock {\em Algebra Universalis}, 63(4):367--403, 2010.

\bibitem{Ba:GAC}
R.~Baer.
\newblock Groups with abelian central quotient group.
\newblock {\em Trans. Amer. Math. Soc.}, 44(3):357--386, 1938.

\bibitem{BBBKZ:MT}
L. Barto, Z. Brady, A. Bulatov, M. Kozik and D. Zhuk. 
\newblock Unifying the three algebraic approaches to the CSP via minimal Taylor algebras,  3 (2024), 
\newblock {\em TheoretiCS} 3:1--76, 2024.

\bibitem{BM:SPD}
W.~Bentz and P.~Mayr.
\newblock Supernilpotence prevents dualizability.
\newblock {\em J. Aust. Math. Soc.}, 96(1):1--24, 2014.

\bibitem{Br:CTL}
R.~H. Bruck.
\newblock Contributions to the theory of loops.
\newblock {\em Trans. Amer. Math. Soc.}, 60:245--354, 1946.

\bibitem{Br:SBS}
R.~H. Bruck.
\newblock {\em A survey of binary systems. }
\newblock Vol. 20. Berlin: Springer, 1971.

\bibitem{Bu:NFM}
A.~Bulatov.
\newblock On the number of finite {M}al\cprime tsev algebras.
\newblock In {\em Contributions to general algebra, 13 (Velk\'e Karlovice,
  1999/Dresden, 2000)}, pages 41--54. Heyn, Klagenfurt, 2001.

\bibitem{DV:ENL}
D.~Daly and P.~Vojt\v{e}chovsk\'{y}.
\newblock Enumeration of nilpotent loops via cohomology.
\newblock {\em J. Algebra}, 322(11):4080--4098, 2009.

\bibitem{Fi:CFF}
S.~Fioravanti
\newblock Closed sets of finitary functions between finite fields of coprime order.
\newblock {\em Algebra universalis}, 81(4):Paper No. 52, 14, 2020.

\bibitem{FM:CTC}
R.~Freese and R.~N. McKenzie.
\newblock {\em Commutator Theory for Congruence Modular Varieties}, volume 125
  of {\em London Math. Soc. Lecture Note Ser.}
\newblock Cambridge University Press, 1987.
\newblock Available from
  \verb+http://math.hawaii.edu/~ralph/Commutator/comm.pdf+.
  
 \bibitem{HM:SFA}
 D.~Hobby and R.~McKenzie.
\newblock {\em The structure of finite algebras}.
\newblock Contemporary Mathematics, Volume 76, 1988.

\bibitem{Is:FGT}
I.~M. Isaacs.
\newblock {\em Finite group theory}, volume~92 of {\em Graduate Studies in
  Mathematics}.
\newblock American Mathematical Society, Providence, RI, 2008.

\bibitem{KS:ISSN}
K.~Kearnes and {\'A}.~Szendrei.
\newblock Is supernilpotence super nilpotence?
\newblock {\em Algebra Universalis}, 81(1):Paper No. 3, 10, 2020.

\bibitem{Ma:VLN}
P.~Mayr.
\newblock Vaughan-{L}ee's nilpotent loop of size 12 is finitely based.
\newblock {\em Algebra Universalis}, 85(1):Paper No. 2, 12, 2024.

\bibitem{MW:CBM}
P.~Mayr and P.~Wynne.
\newblock Clonoids between modules.
\newblock {\em Internat. J. Algebra Comput.}, 34(4):543--570, 2024.

\bibitem{MM:SNN}
M.~Moore and A.~Moorhead.
\newblock Supernilpotence need not imply nilpotence.
\newblock {\em J. Algebra}, 535:225--250, 2019.

\bibitem{Mo:HCT}
A.~Moorhead.
\newblock Higher commutator theory for congruence modular varieties.
\newblock {\em J. Algebra}, 513:133--158, 2018.

\bibitem{SS:TCN}
{\v Z}.~Semani{\v s}inov\'a and D.~Stanovsk\'y.
\newblock Three concepts of nilpotence in loops.
\newblock {\em Results Math.}, 78(4):Paper No. 119, 15, 2023.

\bibitem{SV:CTL}
D.~Stanovsk\'y and P.~Vojt{\v e}chovsk\'y.
\newblock Commutator theory for loops.
\newblock {\em J. Algebra}, 399:290--322, 2014.

\bibitem{SV:SG3}
D.~Stanovsk\'y and P.~Vojt{\v e}chovsk\'y.
\newblock Supernilpotent groups and {$3$}-supernilpotent loops.
\newblock {\em J. Algebra Appl.}, 23(9):Paper No. 2450138, 2024.

\bibitem{Ta:MSA}
A.~Tarlecki. (2013). 
\newblock Some nuances of many-sorted universal algebra: A review. 
\newblock {Bulletin of EATCS}, 2(104), 2011.

\bibitem{Va:NPV}
M.~R. Vaughan-Lee.
\newblock Nilpotence in permutable varieties.
\newblock In {\em Universal algebra and lattice theory ({P}uebla, 1982)},
  volume 1004 of {\em Lecture Notes in Math.}, pages 293--308. Springer,
  Berlin, 1983.

\bibitem{Wr:MGL}
C.~R.~B. Wright.
\newblock On the multiplication group of a loop.
\newblock {\em Illinois J. Math.}, 13:660--673, 1969.  
\end{thebibliography}
%\end{document}

\def\cprime{$'$}

\end{document}